\documentclass[12pt]{amsart}
\numberwithin{equation}{section}
\usepackage{amssymb}
\usepackage{amsmath}

\def\R{\mathbb R}
\def\Q{\mathbb Q}
\def\C{\mathbb C}
\def\N{\mathbb N}
\def\Z{\mathbb Z}
\def\H{\mathbb H}
\def\D{\mathbb D}
\def\FF{\mathcal F}
\def\JJ{\mathcal J}
\def\re{\operatorname{Re}}
\def\im{\operatorname{Im}}

\def\arctanh{\operatorname{arctanh}}

\def\dist{\operatorname{dist}}
\newtheorem{lem}{Lemma}[section]
\newtheorem{thm}{Theorem}[section]
\newtheorem*{thma}{Theorem A}

\theoremstyle{definition}
\newtheorem{defin}{Definition}[section]
\theoremstyle{remark}

\title[Baker Domains]{Some examples of Baker domains}
\author{Walter Bergweiler}
\address{Mathematisches Seminar,
Christian--Albrechts--Universit\"at zu Kiel,
Lude\-wig--Meyn--Str.~4, D--24098 Kiel, Germany}
\email{bergweiler@math.uni-kiel.de}
\author{Jian-Hua Zheng}
\thanks{The first author is supported
by a Chinese Academy of Sciences Visiting Professorship for Senior
International Scientists, Grant No.\ 2010 TIJ10, the Deutsche
Forschungsgemeinschaft, Be 1508/7-2, 
and the ESF Networking Programme HCAA. The second author
is 
supported by the NSF of China, Grant No. 10871108.}
\address{Department of Mathematical Sciences, Tsinghua University, P.\ R.\ China}
\email{jzheng@math.tsinghua.edu.cn}

\begin{document}
\begin{abstract}
We construct entire functions with hyperbolic and simply parabolic
Baker domains on which the functions are not univalent.
The Riemann maps from the unit disk to these Baker domains 
extend continuously to certain arcs on the unit circle.
The results answer questions posed by Fagella and Henriksen,
Baker and Dom\'inguez, and others.
\end{abstract}

\maketitle
\section{Introduction and main results}\label{intr}
The \emph{Fatou set} $\FF(f)$ of an entire function $f$ is the subset
of the complex plane $\C$ where the iterates $f^n$ of $f$ form
a normal family. Its complement $\JJ(f)=\C\setminus\FF(f)$ is the
\emph{Julia  set}. The connected components of $\FF(f)$ are called
\emph{Fatou components}. As in the case of rational functions
there exists a classification of periodic Fatou components 
(cf.~\cite[Section~4.2]{Bergweiler1993}), 
the new feature for transcendental functions being
Baker domains. By definition, a periodic Fatou component $U$ 
is called a \emph{Baker domain} if $f^n|_U\to\infty$ as $n\to\infty$.
The first example of a Baker domain was already given by 
Fatou~\cite[Exemple~I]{Fatou1926} who proved that $f(z)=z+1+e^{-z}$ has
a Baker domain $U$ containing the right halfplane.
Since then many further examples have been given; see~\cite{Rippon2008}
for a survey.

By a result of Baker~\cite{Baker1975}, the domains today named after him are
simply connected. 
A result of Cowen~\cite{Cowen1981} then leads to a
classification of Baker domains, which has turned out to be very
useful.
We introduce this classification following Fagella and 
Henriksen~\cite{FagellaHenriksen},
but note that there are a number of equivalent ways to state it;
see section~\ref{classification} for a detailed discussion.
For simplicity, and without loss of generality,
we consider only the case of an invariant Baker 
domain; that is, we assume that $f(U)\subset U$.
We define an equivalence relation on $U$ by saying that 
$u,v\in U$ are equivalent if there exist $m,n\in\N$ such
that $f^m (u)=f^n(v)$ and we denote by $U/f$ the set of equivalence
classes.
The result of Fagella and Henriksen is the 
following~\cite[Proposition~1]{FagellaHenriksen}.
\begin{thma}
Let $U$ be an invariant Baker domain of the entire function~$f$.
Then $U/f$ is a Riemann surface conformally equivalent to
exactly one of the following cylinders:
\begin{itemize}
\item[(a)] $\{z\in\C: -s<\im z< s\}/\Z$, for some $s>0$;
\item[(b)] $\{z\in\C: \im z> 0\}/\Z$;
\item[(c)] $\C/\Z$.
\end{itemize}
\end{thma}
We call $U$ \emph{hyperbolic} in case (a), \emph{simply parabolic} in case
(b) and \emph{doubly parabolic} in case (c).
More generally, the above classification
 holds when $U$ is a simply connected domain,
$U\neq\C$, and $f:U\to U$ is holomorphic without fixed point in~$U$.

While we defer a detailed discussion of this classification to
section~\ref{classification}, we indicate where the names for the different types
of Baker domains come from.
Since $U$ is simply connected, there exists 
a conformal map $\phi:\D\to U$, where $\D$ is the unit disk.
Then $g=\phi^{-1}\circ f\circ \phi$
maps $\D$ to $\D$. (In fact, it can be shown that $g$ is an
inner function.) By the Denjoy-Wolff-Theorem, there exists
$\xi\in\overline{\D}$ such that $g^n\to\xi$  as $n\to\infty$.
As $f$ has no fixed point in $U$ and thus $g$ has no fixed
point in $\D$, we actually find that $\xi\in\partial \D$.
Now suppose that $g$ extends analytically to a neighborhood of $\xi$.
Then $g(\xi)=\xi$ and
$U$ is hyperbolic if $\xi$ is an attracting fixed point,
$U$ is simply parabolic if $\xi$ is a parabolic point with one petal and
$U$ is doubly parabolic if $\xi$ is a parabolic point with two petals.

We note that the Baker domain $U$ of
Fatou's example  $f(z)=z+1+e^{-z}$ mentioned above is doubly parabolic.
For example, this follows directly from Lemma~\ref{lemma2.2} below.
It is known (see the remark after
Lemma~\ref{lemma2.3} below) that if $U$ is a doubly 
parabolic Baker domain, then $f|_U$ is not univalent.
Equivalently, $U$ contains a singularity of the inverse 
function of $f:U\to U$.
(For Fatou's example 
it can be checked directly that $f:U\to U$ is not univalent,
as $U$ contains the critical points $2\pi i k$, with $k\in\Z$.)
On the other hand,  $f|_U$  may be univalent in a hyperbolic
or simply parabolic Baker domain. The first examples of
Baker domains where the function is univalent where given by
Herman~\cite[p.~609]{Herman1985} and 
Eremenko and Lyubich~\cite[Example~3]{EremenkoLyubich1987}.
We mention some further examples of Baker domains:
\begin{itemize}
\item[(1)]
simply parabolic Baker domains in which the
function is univalent (\cite[Theorem~3]{BakerWeinreich},
\cite[Section~5.3]{BaranskiFagella} and
\cite[Section~4, Example~2]{FagellaHenriksen});

\item[(2)]
 hyperbolic Baker domains in which
the function is univalent (\cite[Sections~5.1 and~5.2]{BaranskiFagella},
\cite[Theorem~1]{Bergweiler1995} and
 \cite[Section~4, Example~1]{FagellaHenriksen};

\item[(3)]
 hyperbolic Baker domains in which
the function is not univalent (\cite[Theorem~3]{Rippon},
and \cite[Theorems~2 and~3]{RipponStallard};
\end{itemize}

We note that there are no examples of simply parabolic Baker domains in 
which the function is not univalent.
Therefore it was asked in~\cite[Section~4]{FagellaHenriksen}
and~\cite[p.~203]{Zheng1}
whether such domains actually exist. 
Our first result says that this is in fact the case.
\begin{thm} \label{thm1}
There exists an entire function $f$ with a simply parabolic
Baker domain in which $f$ is not univalent.
\end{thm} 
It turns out 
that the function constructed also provides
an answer to a question about the boundary of
Baker domains which arises from
the work of Baker and Weinreich~\cite{BakerWeinreich},
Baker and Dom\'inguez~\cite{BakerDominguez},
Bargmann~\cite{Bargmann}
and Kisaka~\cite{Kisaka1997,Kisaka1998}.

For a conformal map
$\phi:\D\to U$ let $\Xi$ be the set of $\xi\in\partial \D$
such that $\infty$ is contained in the 
cluster set of $\phi$ at $\xi$ and let $\Theta$ 
be the set of $\xi\in\partial \D$ such that 
$\lim_{r\to 1} \phi(r\xi)=\infty$. Clearly, $\overline{\Theta}\subset \Xi$.
The sets $\Theta$ and $\Xi$ depend on the choice of the conformal
map $\phi$, but
we will only be concerned with the question whether $\Xi$ is
equal to $\partial \D$ or $\Theta$ is dense in  $\partial \D$,
and these statements are independent of~$\phi$.

Devaney and Goldberg~\cite{DevaneyGoldberg} showed that 
if $\lambda\in\C\setminus\{0\}$ is such that $f(z)=\lambda e^z$
has an attracting fixed point and $U$ denotes its attracting
basin, then $\Theta$ is dense in $\partial \D$.
Baker and Weinreich~\cite[Theorem~1]{BakerWeinreich}
proved that if $f$ is an arbitrary transcendental entire function
and $U$ is an unbounded invariant Fatou component  of $f$
which is not a Baker domain and which thus -- by the classification
of periodic Fatou components -- is an attracting or parabolic basin
or a Siegel disk, then $\Xi=\partial \D$. 
Baker and Dom\'inguez~\cite[Theorem~1.1]{BakerDominguez} showed, under the same
hypothesis, that if $\Theta\neq \emptyset$, then 
$\overline{\Theta}=\partial \D$.
Under additional hypotheses this had been proved before by 
Kisaka~\cite{Kisaka1997,Kisaka1998}.  

As shown by Baker and Weinreich~\cite[Theorem~3]{BakerWeinreich}, 
the above results need not hold for Baker 
domains: they gave an example of a Baker domain bounded by
a Jordan curve on the sphere. Clearly, for this example the sets
$\Xi$ and $\Theta$ consist of only one point. 
On the other hand,
Baker and Weinreich~\cite[Theorem~4]{BakerWeinreich} showed
that if a Baker domain $U$ is bounded by a Jordan curve on the 
sphere, then $f$ is univalent in~$U$.

If $U$ is a Baker
domain, then $\infty$ is accessible in $U$ and thus we always
have $\Theta\neq \emptyset$ for Baker domains.
Baker and Dom\'inguez~\cite[Theorem~1.2]{BakerDominguez}
showed that if $U$ is a Baker domain where $f$ is not univalent, then 
$\overline{\Theta}$ contains a  perfect subset of $\partial \D$.
Again, this had been proved before by Kisaka~\cite{Kisaka1997,Kisaka1998}
under additional hypotheses.

Baker and Dom\'inguez~\cite[p.~440]{BakerDominguez} asked whether
even
 $\overline{\Theta}=\partial \D$ if $U$ is a Baker domain where
$f$ is not univalent. It was shown by 
Barg\-mann~\cite[Theorem~3.1]{Bargmann} 
that this is in fact the case for doubly parabolic Baker domains.
With $\overline{\Theta}$ replaced by $\Xi$ this result
appears in~\cite{Kisaka1997}.
The question whether these results also hold for
hyperbolic and simply parabolic  Baker domains was left open
in these papers.

It turns out 
that the Baker domain constructed in Theorem~\ref{thm1}
can be chosen such that 
\begin{equation}\label{Xi}
\Xi\neq \partial \D.
\end{equation}
In particular, this implies that $\overline{\Theta}\neq \partial \D$.
A modification of the method also yields an example of a hyperbolic 
Baker domain with this property. We thus have the following result.

\begin{thm} \label{thm2}
There exists an entire function $f_1$ with a simply parabolic
Baker domain $U_1$ such that $f_1|_{U_1}$ is not univalent and 
the set $\Xi$ defined above satisfies~\eqref{Xi}.

There also exists an entire function $f_2$ with a hyperbolic
Baker domain $U_2$ satisfying~\eqref{Xi} such that $f_2|_{U_2}$
is not univalent.
\end{thm} 

\section{Classification of Baker domains}\label{classification}
We describe the classification given by Cowen~\cite{Cowen1981} following
K\"onig~\cite{Koenig1999} and Barg\-mann~\cite{Bargmann};
see also~\cite[Section~5]{Rippon2008} for a discussion of this
classification.
Let $U$ be a domain and let $f:U\to U$ be holomorphic.
We say that a subdomain $V$ of $U$ is \emph{absorbing} for~$f$,
if $V$ is simply connected, 
$f(V)\subset V$
and for any compact subset $K$ of $U$ there exists $n=n(K)$ such
that $f^{n}(K)\subset V$.
(Cowen used the term \emph{fundamental} instead of absorbing.)
Let $\H=\{z\in\C:\re z>0\}$ be the right halfplane.

\begin{defin}\label{def2.1} 
Let $f:U\to U$ be holomorphic.
Then $(V, \varphi, T, \Omega)$ is called an
\emph{eventual  conjugacy} of $f$ in $U$, if the following statements
hold:
\begin{itemize}
\item[(i)] $V$ is  absorbing for $f$;

\item[(ii)] $\varphi:U\rightarrow \Omega\in\{\mathbb{H},\mathbb{C}\}$ is
holomorphic and $\varphi$ is univalent in $V$;

\item[(iii)] $T$ is a M\"obius transformation mapping $\Omega$
onto itself and $\varphi(V)$ is absorbing for $T$;

\item[(iv)] $\varphi(f(z))=T(\varphi(z))$ for $z\in U$.
\end{itemize}
\end{defin}
K\"onig~\cite{Koenig1999} used
the term \emph{conformal} conjugacy.
With the terminology \emph{eventual} conjugacy we have followed
Barg\-mann~\cite{Bargmann}.

If $U$ is the basin of an attracting 
(but not superattracting) fixed point~$\xi$, then 
an eventual conjugacy is given 
 by the solution of
the Schr\"oder-K\oe nigs functional equation
$\varphi(f(z))=\lambda \varphi(z)$,
with $\lambda=f'(\xi)\neq 0$.
Similarly, eventual conjugacies in parabolic domains are
given by the solutions of Abel's functional equation
$\varphi(f(z))=\varphi(z)+1$.

In what follows, we assume that $f$ has no fixed points in~$U$.
Clearly, this is the case for Baker domains~$U$.
The result of Cowen can now be stated as follows.
\begin{lem}\label{lemma2.1} 
Let $U\neq\C$ be a simply connected domain and
$f:U\to U$ a holomorphic function  without fixed point in~$U$. 
Then $f$ has an eventual conjugacy
$(V, \varphi, T, \Omega)$.
Moreover, $T$ and $\Omega$ may be chosen as
exactly one of the following possibilities:
\begin{itemize}
\item[(a)] $\Omega=\H$ and $T(z)=\lambda z$, where $\lambda >1$;
\item[(b)] $\Omega=\H$ and $T(z)=z+i$ or $T(z)=z-i$;
\item[(c)] $\Omega=\C$ and $T(z)=z+1$.
\end{itemize}
\end{lem}
It turns out 
(cf.~\cite{FagellaHenriksen})
that the cases listed in Lemma~\ref{lemma2.1}
correspond precisely to the cases of Theorem A. Thus 
$U$ is hyperbolic in case~(a), simply parabolic in case~(b) 
and doubly parabolic in case~(c).

K\"onig~\cite[Theorem 3]{Koenig1999} has given the 
following useful characterization of the different cases.
\begin{lem}\label{lemma2.2}
Let $U\neq\C$ be an unbounded simply connected domain and
$f:U\to U$ a holomorphic function such that $f^n |_U\to  \infty$
as $n\to\infty$. 
For $w_0\in U$ put
$$
w_n=f^n(w_0) 
\quad\text{and}\quad
d_n=\dist(w_n,\partial U).
$$
Then
\begin{itemize}
\item[(a)] $U$ is hyperbolic if
there exists $\beta>0$ such that 
$|w_{n+1}-w_n|/d_n\geq \beta$ for all $w_0\in U$ and all $n\geq 0$;
\item[(b)] $U$ is simply parabolic if 
$ \liminf_{n\to\infty}|w_{n+1}-w_n|/d_n>0$
for all $w_0\in U$, but
$$\inf_{w_0\in U} \limsup_{n\to\infty}\frac{|w_{n+1}-w_n|}{d_n}=0 ;$$
\item[(c)]  $U$ is doubly parabolic if
$ \lim_{n\to\infty}|w_{n+1}-w_n|/d_n=0$
for all $w_0\in U$.
\end{itemize}
\end{lem}
Denote by $\rho_U(\cdot,\cdot)$ the hyperbolic metric in
a hyperbolic domain $U$,
using the normalization where the density $\lambda_\D$ 
in the unit disk is given by $\lambda_\D(z)=2/(1-|z|^2)$.
Considering the hyperbolic metric
instead of the Euclidean metric in Lemma~\ref{lemma2.2}
leads to the following result; see~\cite[Lemma~2.6]{Bargmann}
or~\cite[Theorem 2.2.11]{Zheng1}.
\begin{lem}\label{lemma2.3}
Let $U\neq\C$ be a simply connected domain and
$f:U\to U$ a holomorphic function  without fixed point in~$U$. 
For $z\in U$ put 
$$\ell(z)=\lim_{n\to\infty}\rho_U(f^{n+1}(z),f^n(z)).$$
Then
\begin{itemize}
\item[(a)] $U$ is hyperbolic if
$\inf_{z\in U}\ell(z)>0$;
\item[(b)] $U$ is simply parabolic if 
$\ell(z)>0$ for all $z\in U$, but
$\inf_{z\in U}\ell(z)=0$;
\item[(c)]  $U$ is doubly parabolic if
$\ell(z)=0$ for all $z\in U$.
\end{itemize}
\end{lem}
Note that the sequence 
$\left( \rho_U(f^{n+1}(z),f^n(z))\right)_{n\in\N}$
is non-increasing by the Schwarz-Pick Lemma (Lemma~\ref{lemmaSP}
below). Thus the
limit defining $\ell(z)$ exists.

Let now $U$ be an invariant Baker domain of the entire function~$f$.
It is not difficult to see that if $f|_U$ is univalent, then
$f(U)=U$.
Also, the Schwarz-Pick Lemma says that if $f|_U$ is 
univalent, then
$$\rho_U(f^{n+1}(z),f^n(z))=\rho_U(f(z),z)$$
for all $n\in\N$ and $z\in U$.
It now follows from 
Lemma~\ref{lemma2.3}, as already mentioned in the introduction,
 that $U$ cannot be doubly parabolic if $f|_U$ is univalent.

\begin{lem}\label{lemma2.5}
Let $f: U\to U$ be as in Lemma~\mbox{\rm\ref{lemma2.1}} and
let
$U_0$ be an absorbing domain for~$f$.
Then $f:U_0\to U_0$ and $f: U\to U$ are
 of the same type according to the
classification given in Theorem~A or Lemma~\mbox{\rm\ref{lemma2.1}}.
\end{lem}
\begin{proof}
Let $(V,\varphi,T,\Omega)$ be an
eventual  conjugacy of $f$ in~$U$.
Since,  by the definition
of an absorbing domain,  $V$ and $U_0$ are simply connected, 
the components of  $V\cap U_0$ 
are also simply connected.
It was shown by Cowen~\cite[p.~79-80]{Cowen1981} (see 
also~\cite[Lemma~2.3]{Bargmann}) that there exists a
component $W$ of $V\cap U_0$ which is absorbing for 
$f$ in~$U$. Moreover, $\varphi(W)$ is absorbing for $T$
in~$\Omega$. Thus
$(W,\varphi|_W,T,\Omega)$ is an 
eventual  conjugacy of  both $f:U\to U$
and $f:U_0\to U_0$.
The conclusion follows.
\end{proof}
We note that a somewhat different approach to  classifying holomorphic
self-maps of $\H$, based on the sequences 
$(\rho_\H(f^{n+1}(z),f^n(z)))_{n\in\N}$,
 was developed
by Baker and Pommerenke~\cite{BakerPommerenke,
Pommerenke1979}; see also~\cite{Bonfert}. 
As shown by K\"onig~\cite[Lemma~3]{Koenig1999}, this  leads to
the same classification as above. 

We mention that Baker domains may also be defined for 
functions meromorphic in the plane.
In general, Baker domains of meromorphic
functions are  multiply connected.
K\"onig~\cite[Theorem~1]{Koenig1999}
has shown that if $U$ is a Baker domain of a meromorphic function  
with only
finitely many poles, then an eventual conjugacy in $U$ exists
and the conclusion of Lemma~\ref{lemma2.2} holds.
On the other hand, eventual conjugacies  need not exist for Baker domains
of meromorphic functions with infinitely many 
poles~\cite[Theorem~2]{Koenig1999}.
For further results on Baker domains of meromorphic functions
we refer to~\cite{Zheng1}. In particular, we note 
that if $U$ is a multiply connected Baker domain of a 
meromorphic function $f$ such that there exists an eventual
conjugacy in~$U$, then $U$ contains at
least two singularities of~$f^{-1}$; see~\cite[p.~200]{Zheng1}. 
Thus $f|_U$ is not 
univalent in a multiply connected Baker domain~$U$
with 
an eventual conjugacy.

More generally, one may also consider 
functions meromorphic outside a small set. 
For such functions the classification of Baker domains
is discussed in~\cite[Section~4]{Zheng}.

The classification of Baker domains mentioned above appears
in various other questions related to Baker domains.
Besides the papers already cited 
 we mention~\cite{Bergweiler2001,BergweilerDrasinLangley,
BuffRueckert,Lauber}.

\section{Preliminary lemmas}\label{prelims}
The following result, already used in section~\ref{classification},
 is known as the Schwarz-Pick Lemma; see, e.g., \cite[p.12]{Steinmetz}.
\begin{lem}\label{lemmaSP}
Let $U$ and $V$  be simply connected hyperbolic 
domains and $f:U\to V$ be holomorphic.
Then 
$\rho_V(f(z_1), f(z_2))\leq
\rho_U(z_1, z_2)$ 
for all $z_1,z_2\in U$.
If there exists $z_1,z_2\in U$ with $z_1\neq z_2$ such that
$\rho_V(f(z_1), f(z_2))=\rho_U(z_1, z_2)$, then $f$ is bijective.
\end{lem}
If $U\subset V$, then we may apply the Schwarz-Pick Lemma to
$f(z)=z$ and obtain $\rho_V(z_1), z_2)\leq
\rho_U(z_1, z_2)$ for all $z_1,z_2\in U$.

The following result is the 
analogue of the Schwarz-Pick Lemma for
quasiconformal mappings~\cite[Section II.3.3]{LehtoVir}.
Note that a different normalization of the hyperbolic
metric is used in~\cite{LehtoVir}.

\begin{lem}\label{lem1.1}  Let 
$U$ and $V$  be simply connected hyperbolic 
domains and $f:U\to V$ be a $K$-quasiconformal mapping.
Then 
\begin{equation}
\label{hp}
\rho_V(f(z_1), f(z_2))\leq  
M_K(\rho_U(z_1, z_2))
\end{equation}
for $z_1,z_2\in U$, with 
\begin{equation}
\label{hp2}
M_K(x)=2\operatorname{arctanh}\left(\varphi_K
\left( \tanh\tfrac12 x\right)\right).
\end{equation}
Here $\varphi_K$ is the Hersch-Pfluger distortion function.
\end{lem}
The function $\arctanh\circ\; \varphi_K\circ \tanh$
appearing on the right hand side of~\eqref{hp2} has been studied in detail
in a number of papers. For example, it was shown 
in~\cite[Theorem~1.6]{QVV} that
this function is strictly increasing and concave.
Various estimates of this function in terms of elementary functions
are known. We only mention~\cite[Theorem~11.2]{Vuorinen88} where
it was shown
that the conclusion of Lemma~\ref{lem1.1} holds with~\eqref{hp} replaced 
by $\rho_V(f(z_1), f(z_2))\leq K\left(\rho_U(z_1, z_2)+\log 4\right)$.
We do not need any explicit estimate for $M_K$, but 
the fact that~\eqref{hp} holds for some 
non-decreasing function $M_K:[0,\infty)\to [0,\infty)$
suffices.

The following result~\cite[Lemma~1]{Shishikura1987} is the fundamental
lemma for quasiconformal surgery.
Here $\overline{\C}=\C\cup\{\infty\}$ denotes the Riemann sphere.
\begin{lem}\label{lem1.3}
Let $g:\overline{\C}\to\overline{\C}$ be a quasiregular mapping.
Suppose that there are disjoint open subsets $E_1,\dots,E_m$ of $\overline{\C}$,
quasiconformal mappings $\Phi_i:E_i\to E_i'\subset \overline{\C}$,
for $i=1,\dots,m$, and an integer $N\geq 0$ satisfying the following
conditions:
\begin{itemize}
\item[(i)] $g(E)\subset E$ where $E=E_1\cup\dots\cup E_m$;
\item[(ii)] $\Phi\circ g\circ \Phi_i^{-1}$ is analytic in
$E_i'=\Phi_i(E_i)$, where $\Phi:E\to\overline{\C}$ is defined
by $\Phi|_{E_i}=\Phi_i$;
\item[(iii)]  
$g_{\overline{z}}=0$ a.e.\ on $\overline{\C}\setminus g^{-N}(E)$.
\end{itemize}
Then there exists a quasiconformal mapping $\psi:\overline{\C}\to\overline{\C}$
such that $\psi\circ g\circ \psi^{-1}$ is 
a rational function. Moreover, $\psi\circ \Phi_i^{-1}$ is conformal
in $E_i'$ and $\psi_{\overline{z}}=0$ a.e.\ on 
$\overline{\C}\setminus \bigcup_{n\geq 0} g^{-n}(E)$.
\end{lem}
Shishikura stated the 
result in~\cite{Shishikura1987} 
for rational functions, but
it holds for entire functions as well. A stronger result,
stated for entire functions,
can be found in \cite[Theorem~3.1]{KisakaShi}.

\section{Proof of Theorem \ref{thm1}}\label{proof1}
Put
$g(z)=e^{2\pi i\alpha}ze^z$, where $\alpha\in[0,1]\setminus\Q$ is 
chosen such that $g$ has a Siegel disk $S$ at $0$,
and put
$$h(z)=2\pi i(\alpha +m)+z+e^z$$
where $ m\in\mathbb{Z}$.
Using $g(e^z)=\exp h(z)$ it can be shown that
$h$ has a Baker domain $V=\log S$
on which $h$ is univalent.
This example, with $m=0$, is due to Herman~\cite[p.\ 609]{Herman1985};
see also~\cite{BakerWeinreich,Bergweiler1995,Bergweiler1995a}.
 We shall assume, however, that
$m\geq  3$. It is not difficult to see that $V$ is simply parabolic.

There exists $r_0\in (0,1)$ and a $g$-invariant domain $S_0\subset S$
such that 
$$D(0,r_0)\subset S_0\subset D(0,1).$$
 Here and in the 
following $D(a,r)$ denotes the open disk of radius $r$ around a
point~$a$.
With $x_0=\log r_0$ we 
thus see that $V$ contains 
$H_0=\{z\in\mathbb{C}:\re <x_0\}$.
Moreover, if $z\in H_0$, then 
\begin{equation}\label{Imhn}
\im h^n(z)\geq 2\pi(\alpha+m)+\im h^{n-1}(z)-1> \im z+2n\pi
\quad\text{and}\quad
\re h^n (z)<0
\end{equation}
for $n\in\N$.

For $x_1< x_0-\pi$ we define
$$S(x_1)=\{z\in\mathbb{C}: \re z<x_1, \;|\im z|< 2\pi\}\cup D(x_1,2\pi)$$
and
$$T(x_1)=\{z\in\mathbb{C}: \re z<x_1, \;|\im z|< \pi\}\cup D(x_1,\pi)$$
We also put
$$k(z)=2\pi i(\alpha+m)+z+\exp(e^{-z}-L),$$
for a large constant $L$ to be determined later.

Now we define $F:\mathbb{C}\rightarrow\mathbb{C}$ as follows.
For $z\in\mathbb{C}\setminus S(x_1)$ we put $F(z)=h(z)$, for $z\in
\overline{T(x_1)}$  we put $F(z)=k(z)$ and 
for $z\in S(x_1)\setminus \overline{T(x_1)}$ we define $F(z)$ by interpolation.
Thus for $x\leq x_1$ and $\pi\leq y\leq 2\pi$ we put
$$F(x+iy)=\frac{y-\pi}{\pi}h(x+2\pi i)+\frac{2\pi-y}{\pi}k(x+\pi i),$$
for $x\leq x_1$ and $-2\pi\leq y\leq -\pi$ we put
$$F(x+iy)=\frac{-y-\pi}{\pi}h(x-2\pi i)+\frac{2\pi+y}{\pi}k(x-\pi i),$$
and for $\pi\leq r\leq 2\pi$ and $-\pi/2\leq \varphi\leq \pi/2$
we put
$$F(x_1+re^{i\varphi})= \frac{r-\pi}{\pi}h(x_1+2\pi e^{i\varphi})
+\frac{2\pi-r}{\pi}k(x_1+\pi e^{i\varphi}).$$
We claim that $F$ is quasiregular.
By definition, $F$ is holomorphic in $T(x_1)$ and in $\C\setminus\overline{S(x_1)}$.
So it suffices to consider the dilatation in $S(x_1)\setminus \overline{T(x_1)}$.
We first consider the subregion
$A
=\{x+iy: x\leq x_1, \;\pi<y<2\pi\}$ of $S(x_1)\setminus \overline{T(x_1)}$.
For $z=x+iy\in A$ we have 
\begin{equation}\label{FinOmega1}
F(z)=2\pi i(\alpha+m)+z + P(z)
\end{equation}
with 
\begin{equation}\label{RinOmega1}
P(x+iy)=\frac{y-\pi}{\pi}e^x+\frac{2\pi-y}{\pi}\exp(-e^{-x}-L)).
\end{equation}
It is easy to see that if $|x_1|$ is large enough, then
$|P_x(z)|\leq 1/4$ and 
$|P_y(z)|\leq 1/4$ for $z\in A$. Thus
$|P_z(z)|\leq 1/4$ and 
$|P_{\overline{z}}(z)|\leq 1/4$ and hence 
$|F_z(z)|\geq 3/4$ and $|F_{\overline{z}}(z)|\leq 1/4$ for $z\in A$.
It follows that $F$ is quasiregular in $ A$.
The argument for the domain
$\overline{A}=\{x+iy: x\leq x_1, \; -2\pi<y<-\pi\}$ is analogous.

Now we consider the region
$B=
\{x_1+re^{i\varphi}: \pi\leq r\leq 2\pi,\;  -\pi/2\leq \varphi\leq \pi/2\}$.
For $z=x_1+re^{i\varphi}\in B$ we have 
\begin{equation}\label{FinOmega2}
F(z)=2\pi i(\alpha+m)+z + Q(z)
\end{equation}
with 
\begin{equation}\label{RinOmega2}
\begin{aligned}
Q(x_1+re^{i\varphi})
=&\ \frac{r-\pi}{\pi}\exp\left(x_1+2\pi e^{i\varphi}\right)\\
&\ +\frac{2\pi-r}{\pi}\exp\left(-\exp\left(x_1+\pi e^{i\varphi}\right)-L\right).
\end{aligned}
\end{equation}
The computation of the partial derivatives of $Q$ is more cumbersome
than for $P$, but again it follows 
$|Q_z(z)|\leq 1/4$ 
and $|Q_{\overline{z}}(z)|\leq 1/4$ for $z\in B$
if $|x_1|$ and $L=L(x_1)$ are large enough. 
As before this implies that
$F$ is quasiregular in $ B$.

It follows from~\eqref{FinOmega1}, \eqref{RinOmega1},
\eqref{FinOmega2} and~\eqref{RinOmega2}, together with
the corresponding representation in $\overline{A}$, that
\begin{equation}\label{2.1}
\im F(z)>\im z+2\pi(\alpha+m)-1 >3\pi
\quad\text{and}\quad \re F(z)<x_0
\end{equation}
for $z\in \overline{S(x_1)}\setminus T(x_1)$,
provided $|x_1|$ and $L$ are large enough.
Together with~\eqref{Imhn} this implies that  
if $x_1$ and $L$ are suitably chosen, then
every orbit passes through $\overline{S(x_1)}\setminus  T(x_1)$,
which is  the set where $F$ is not holomorphic, at most once. 

It now follows from Lemma \ref{lem1.3},
applied with $g=F$, $N=1$, $m=1$, 
$$E_1=\bigcup_{n=1}^\infty F^n
\left(S(x_1)\setminus\overline{T(x_1)}\right)$$
 and
 $\Phi_1=\text{id}_{E_1}$,
 that there exists a quasiconformal map
$\psi:\C\to\C$ such that $f=\psi\circ F\circ\psi^{-1}$ is an
entire function.

It is easy to see that $f$ is transcendental.
For example, let $a\in\mathcal{J}(h)$
such that $h^{-1}(a)$ is infinite. The
complete invariance of $\JJ(h)$  yields that
$h^{-1}(a)\cap S(x_1)=\emptyset$ so that $h^{-1}(a)\subset
F^{-1}(a)=\psi^{-1}(f^{-1}(\psi(a)))$. Therefore $f^{-1}(\psi(a))$ is
infinite, which implies that $f$ is transcendental.

In the sequel, we will apply the concepts of the Fatou-Julia
theory also to the quasiregular function~$F$. For example, we can define
$\JJ(F)$ as the set where the iterates of $F$ are not normal 
and find that $\JJ(f)=\psi(\JJ(F))$.

It follows from~\eqref{Imhn} and~\eqref{2.1} that
$$\im F^n(z)\rightarrow\infty\quad\text{for}\ 
z\in V\setminus\bigcup_{k=0}^\infty h^{-k}(T(x_1))$$
as $n\rightarrow\infty$.
With $W_0=\{z\in\C: \re z< x_0,\;  \im z> 2 \pi\}$ we have
$$V\setminus\bigcup_{k=0}^\infty h^{-k}(T(x_1))\supset 
W_0$$
and thus find that $F$ has a Baker domain $W$ containing
$W_0$.
Using~\eqref{2.1} we see that $\overline{S(x_1)}\setminus T(x_1)
\subset W$.

We now show that $T(x_1)\cap \JJ(F)\neq\emptyset$.
In order to do so we note that if $x\in\R$ is sufficiently
large, then
$$
\left| F\left(-x+i\frac{\pi}{2}\right)\right|
=\left| k\left(-x+i\frac{\pi}{2}\right)\right|
=\left|2\pi i(\alpha+m)-x+i\frac{\pi}{2}
+\exp\left(i e^x-L)\right)\right|
\leq 2x
$$
while 
$$
\left| F'\left(-x+i\frac{\pi}{2}\right)\right|
=\left| k'\left(-x+i\frac{\pi}{2}\right)\right|
=\left| 1-ie^x \exp\left(i e^x-L)\right)\right|
\geq e^{x-L}-1
\geq e^{x/2}.
$$
Given $a_1,a_2\in \JJ(F)$ 
it now follows from Landau's theorem that  if $x$ is large 
enough, then there exists
$j\in\{1,2\}$ and
$z\in D(-x+i\pi/2,1)$ such that $F(z)=a_j$.
By the complete invariance of $\JJ(F)$ we thus have
$D(-x+i\pi/2,1)\cap \JJ(F)\neq\emptyset$. In particular,
$T(x_1)\cap \JJ(F)\neq\emptyset$.

Since $F$ has the unbounded Fatou component~$W$, 
a result of Baker~\cite{Baker1975} yields that $F$ has
no multiply connected Fatou components.
Since $\overline{S(x_1)}\setminus T(x_1)\subset W\subset\FF(F)$,
this implies that $T(x_1)$ contains an unbounded
component $\Gamma$ of $\JJ(F)$.

For large $x\in\R$ we consider $w_1=-x-2\pi i,\ w_2=F(w_1)$ and
$w_3=F(w_2)$. Obviously, $w_j\in W$ for $j=1,2,3$. We note
that $w_1$ is ``below" the strip $T(x_1)$ while
$w_2$ and $w_3$ are ``above" this strip by~\eqref{2.1}.
In fact, we have 
$$w_2=2\pi i(\alpha+m-1)-x+e^{w_1}\in W_0$$ 
and 
$$w_3=2\pi i(\alpha+m)+w_2+e^{w_2}=2\pi i(2\alpha+2m-1)-x+e^{w_1}
+e^{w_2}\in W_0.$$
We choose $\delta>0$ such that 
$r_1=2\pi(\alpha+m)+\delta<r_2=2\pi(2\alpha+2m-3)-\delta$
and find that
$$w_2\in D(w_3,r_1)\subset 
 D(w_3,r_2 )$$
for sufficiently large~$x$.
This implies that
$$\rho_{D(w_3,r_2 )}(w_2,w_3)= \rho_{\D}((w_2-w_3)/r_2,0)
\leq 2\arctanh (r_1/r_2).$$
As $D(w_3,r_2 )\subset W_0\subset W$,
the Schwarz-Pick Lemma now yields that
\begin{equation}\label{w2w3}
\rho_W(w_2,w_3)\leq
\rho_{W_0}(w_2,w_3)\leq 
\rho_{D(w_3,r_2 )}(w_2,w_3)\leq 2\arctanh (r_1/r_2)
\end{equation}
for large~$x$.

Next we show that 
\begin{equation}\label{w1w2}
\rho_W(w_1,w_2)
\rightarrow \infty\quad\text{as } x\rightarrow \infty.
\end{equation}
In order to do so, we note that the preimage of $\Gamma$ under $F$ 
contains an unbounded continuum $\Gamma'$ which is contained in 
$\{z\in\C: \re z< x_1,\; y_1<\im z< y_2\}$ for suitable 
$y_1,y_2$ satisfying $y_1<y_2<-2\pi$.
Let now 
$\gamma$ be a curve connecting $w_1$ and $w_2$ in~$W$.
It follows that there exists $x_2\leq x_1$ such that if
$t\leq x_2$, then there  exists $z\in\gamma$,
$\zeta\in\Gamma$ and $\zeta'\in\Gamma'$ such that 
$\re \zeta=\re \zeta'=\re z=t$ and
$y_1<\im \zeta'<\im z< \im \zeta<\pi$.
This implies that the  density  $\lambda_W$ of the hyperbolic metric
in $W$ satisfies
$$\lambda_W(z)\geq \frac{1}{2\dist(z,\partial W)}\geq
\frac{1}{2\min \{|z-\zeta|,|z-\zeta'|\}}
\geq \frac{1}{\pi-y_1}.$$ 
From this we can deduce that if $x<x_2$, then
$$\int_\gamma \lambda_W(z)|dz| \geq \frac{|x-x_2|}{\pi-y_1}.$$
As this holds for all curves $\gamma$ connecting $w_1$ and $w_2$,
we obtain
$$
\rho_W(w_1,w_2)\geq \frac{|x-x_2|}{\pi-y_1},$$
from which~\eqref{w1w2} follows.

Put $U=\psi(W)$. 
Since $f=\psi\circ F\circ\psi^{-1}$  we find that
$U$ is a Baker domain of~$f$. Let $v_j=\psi(w_j)$,
for $j=1,2,3$, and
denote by $K$ the dilatation of $\psi$.
It follows from Lemma~\ref{lem1.1} and~\eqref{w2w3} that
$$\rho_{{U}}(v_2,v_3)\leq M_K(2\arctanh (r_1/r_2)).$$

Suppose now that $f:{U}\rightarrow {U}$ is univalent.
The Schwarz-Pick Lemma yields that
$\rho_{{U}}(v_2,v_3)=\rho_{{U}}(f(v_1),f(v_2))=\rho_{{U}}(v_1,v_2)$.
Noting that $\psi^{-1}$ is also $K$-quasicon\-for\-mal,
we deduce from Lemma~\ref{lem1.1} that
$$\rho_W(w_1,w_2)\leq M_K(\rho_{{U}}(v_1,v_2))
=M_K(\rho_{{U}}(v_2,v_3)).$$
Combining the last  two estimates we obtain
$$\rho_W(w_1,w_2)\leq M_K(M_K(2\arctanh (r_1/r_2))),$$
which contradicts~\eqref{w1w2}. 
Thus $f:{U}\rightarrow {U}$ is not univalent.

It remains to show that $U$ is a simply parabolic Baker domain.
In order to do so we recall that 
$V$ is a simply parabolic Baker domain of~$h$. We also note that it 
follows from the construction
of $V$ that there exists an absorbing domain $V_0$ of $h$ in $V$
satisfying $V_0 \subset \{z\in\mathbb{C}:\re z>2\pi\}$.
Clearly, $V_0$ is
also  an absorbing domain of $F$ in~$W$.
Hence
$U_0=\psi(V_0)$ is an absorbing domain of $f$ in~$U$. 

Since $\psi$ is analytic in $V_0$, we find 
for $v\in V$ and $w=\psi(v)\in \psi(V)=U$ and large $n$ that
$$\rho_{U_0}(f^{n+1}(w),f^{n}(w))
=\rho_{\psi(V_0)}(\psi(h^{n}(v)),\psi(h^{n+1}(v)))
=\rho_{V_0}(h^{n+1}(v),h^{n}(v)).$$
Since $V$ is a simply parabolic Baker domain of $h$, 
Lemma~\ref{lemma2.5} now yields that $U$ is simply parabolic.
This completes the proof of Theorem~\ref{thm1}.

\section{Proof of Theorem \ref{thm2}}\label{proof2}
Let $\alpha$, $g$, $h$, $\psi$ and $f$ be as in the proof of Theorem \ref{thm1}.
Baker and Weinreich~\cite[Theorem~3]{BakerWeinreich} used 
results of Ghys~\cite{Ghys} and Herman~\cite{Herman}
to show that for suitably chosen $\alpha$ 
the boundary of the Siegel disk $S$ of $h$ is a Jordan curve.
Actually, 
by a  recent result of Zakeri~\cite{Zakeri}, this is the case
if $\alpha$ has bounded type.
Thus in this case the Baker domain $V$ of $h$ is bounded by a Jordan 
curve on the sphere.
Let $\gamma$ be a Jordan arc in $\partial V\cap\{z:\im z>2\pi\}$.
It follows from the construction 
that $\gamma\subset\partial W$ and that the points of $\gamma$ 
are accessible from~$W$.
Thus $\partial U$ contains the Jordan arc $\psi(\gamma)$ consisting of 
points accessible from~$U$. This implies that
$\Xi \not=\partial\D$.
Thus $f_1=f$ and $U_1=U$ have the desired property.

The construction of $f_2$ and $U_2$ is similar.
Here our starting point are the functions
 $g(z)=\frac{1}{2}z^2e^{2-z}$
and $h(z)=2-\log 2+2z-e^z$  considered in~\cite{Bergweiler1995}.
The function $g$ has a superattracting basin $B$ at the origin which
is bounded by a Jordan curve and $V=\log B$ is a hyperbolic
Baker domain of $h$ where $h$ is univalent.
By a similar reasoning as in the proof of Theorem~\ref{thm1}
we will now use quasiconformal surgery to construct an
entire function $f_2$ with a Baker domain $U_2$ where 
$f$ is not univalent.

Here we put, for large a positive integer $M$,
$$S(M)=\{z\in\mathbb{C}:| \re z+ 2\pi M|<2\pi, \;
\im z< -2\pi\}\cup D(-2\pi M-2\pi i,2\pi)$$
and
$$T(M)=\{z\in\mathbb{C}: 
| \re z+ 2\pi M|<\pi, \; \im z< -2\pi\}\cup D(-2\pi M-2\pi i,\pi).$$
Next we put
$k(z)=2-\log 2+2z +\exp(e^{-iz}-L)$
for a large constant~$L$ and define 
$F:\C\to\C$ by $F(z)=h(z)$ for $z\in \C\setminus S(M)$,
by $F(z)=k(z)$ for $z\in \overline{T(M)}$ and by interpolation
in $S(M)\setminus  \overline{T(M)}$.
Similarly as in the proof of Theorem~\ref{thm1} we see that $F$ is 
quasiregular if $M$ and $L=L(M)$ are large enough
and that there exists a quasiconformal map $\psi$ 
such that $f_2=\psi\circ F\circ \psi^{-1}$ is a transcendental
entire function.
Noting that $V\cap\H$ is invariant under $h$ and thus under $F$
we see that $F$ has a Baker domain $W$ containing $V\cap\H$.
Thus $U_2=\psi(W)$ is a Baker domain of~$f_2$, and the construction
shows that $\partial U_2$ contains a Jordan arc consisting
of points accessible from~$U_2$.

To show that $U_2$ is hyperbolic we use again
Lemma~\ref{lemma2.5}, noting that the domain
$V_0=\{z\in\C: \re z< 3\pi M\}$
is absorbing for $h$ and $F$ in $V$ and that $\psi$ is analytic in~$V_0$.
Finally, the reasoning that $f_2$ is not univalent in $U_2$
is similar to that in the proof of Theorem~\ref{thm1}.


\begin{thebibliography}{20}

\bibitem{Baker1975} I.\ N.\ Baker,
The domains of normality of an entire function.
{\rm Ann.\ Acad.\ Sci.\ Fenn.\ Ser.\ A I Math.}
1 (1975), 277--283.

\bibitem{BakerDominguez} I. N. Baker and Dom\'inguez, 
Boundaries of unbounded Fatou components of entire functions.
Ann. Acad. Sci. Fenn. Math. 24 (1999), 437–-464. 

\bibitem{BakerPommerenke}
 I.\ N.\ Baker and Ch.\ Pommerenke,
On iteration of analytic functions in a halfplane. II.
J.~London Math.\ Soc.\ (2) 20 (1979), 255--258.

\bibitem{BakerWeinreich}
 I.\ N.\ Baker and J.\ Weinreich,
Boundaries which arise in the iteration of entire functions.
{\rm Rev.\ Roumaine Math.\ Pures Appl.}
36 (1991), 413--420.

\bibitem{BaranskiFagella} 
K. Baranski and N. Fagella, 
Univalent Baker domains.
 Nonlinearity 14 (2001), 411--429.

\bibitem{Bargmann}
D.\ Bargmann,
Iteration of inner functions and boundaries of components of the 
Fatou set.
In ``Transcendental Dynamics and Complex Analysis''.
London Math.\ Soc.\ Lect.\ Note Ser.\ 348.
Edited by P.\ J.\ Rippon and G.\ M.\ Stallard,
Cambridge Univ.\ Press, Cambridge, 2008, 1--36.

\bibitem{Bergweiler1993}
W.\ Bergweiler,
Iteration of meromorphic functions.
{\rm Bull.\ Amer.\ Math.\ Soc.\ (N.\ S.)}
29 (1993), 151--188.

\bibitem{Bergweiler1995} W. Bergweiler, Invariant domains and
singularities. Math. Proc. Camb. Phil. Soc. 117 (1995), 525--536.

\bibitem{Bergweiler1995a} W. Bergweiler, 
On the Julia set of analytic self--maps of the punctured plane.
Analysis 15 (1995), 251--256.


\bibitem{Bergweiler2001} W. Bergweiler, 
Singularities in Baker domains.
Comput. Methods Funct. Theory 1 (2001), 41-–49. 

\bibitem{BergweilerDrasinLangley}
W.\ Bergweiler, D.\ Drasin and J.\ K.\ Langley,
Baker domains for Newton's method.
Ann.\ Inst.\ Fourier 57 (2007), 803–-814. 


\bibitem{Bonfert}
P.\ Bonfert,
On iteration in planar domains.
Michigan Math. J. 44 (1997), 47--68.

\bibitem{BuffRueckert}
X. Buff and J.\ R\"uckert,
Virtual immediate basins of Newton maps and asymptotic 
values. Int. Math. Res. Not. 2006, 65498, 1–-18. 


\bibitem{Cowen1981} C. C. Cowen,
 Iteration and the solution of functional equations for 
functions analytic in the unit disk. 
Trans.\ Amer.\ Math.\ Soc. 265 (1981), 69–-95.

\bibitem{DevaneyGoldberg} R.\ L.\ Devaney and L.\ Goldberg,
 Uniformization of attracting basins for exponential maps. 
Duke Math. J. 55 (1987), 253–-266.


\bibitem{EremenkoLyubich1987}
 A.\ E.\ Eremenko and M.\ Yu.\ Lyubich,
Examples of entire functions with pathological dynamics.
{\rm J.\ London Math.\ Soc.} (2)
36 (1987), 458--468.


\bibitem{FagellaHenriksen} N. Fagella and C. Henriksen, Deformation
of entire functions with Baker domains. 
Discrete Cont.\ Dynam.\ Systems 
15 (2006), 379--394.

\bibitem{Fatou1926} P.\ Fatou,
Sur l'it\'eration des fonctions transcendantes
enti\`eres. 
{\rm Acta Math.} 47 (1926), 337--360.

\bibitem{Ghys}
E. Ghys,
Transformations holomorphes au voisinage d'une courbe de Jordan.  
C. R. Acad. Sci. Paris S\'er. I Math. 298 (1984),  385--388.

\bibitem{Herman}
 M.\ Herman, 
Conjugaison quasi sym\'etrique des diff\'eomorphismes du cercle \`a
 des rotations et applications aux disques singuliers de Siegel, I. 
Unpublished manuscript, available at
http://www.math.kyoto-u.ac.jp/$\sim$mitsu/Herman/index.html


\bibitem{Herman1985}
 M.\ Herman,
Are there critical points on the boundary of singular domains?
{\rm Comm.\ Math.\ Phys.}
99 (1985), 593--612.

\bibitem{Kisaka1997} M.\ Kisaka,
On the connectivity of Julia sets of transcendental entire functions.
Sci. Bull. Josai Univ. 1997, Special issue no.~1, 77--87.

\bibitem{Kisaka1998} M.\ Kisaka,
On the connectivity of Julia sets of transcendental entire functions.
Ergodic Theory Dynam. Systems 18 (1998), 189--205.


\bibitem{KisakaShi} M.\ Kisaka and M.\ Shishikura,
On multiply connected wandering domains of entire functions. In
``{\rm Transcendental dynamics and complex analysis}'', edited by
P.\ J.\ Rippon and G.\ M.\ Stallard, LMS Lecture Note Series 348,
Cambridge University Press, 2008, 217--250.

\bibitem{Koenig1999} H. K\"onig, Conformal conjugacies in Baker domains.
 J.\ London Math.\ Soc.\ (2) 59 (1999), 153--170.

\bibitem{Lauber}
A.\ Lauber,
Bifurcations of Baker domains.
 Nonlinearity 20 (2007), 1535--1545.

\bibitem{LehtoVir} 
O.\ Lehto and K.\ Virtanen, Quasiconformal mappings
in the plane. Springer-Verlag, Berlin, 1973.

\bibitem{Pommerenke1979}
Ch.\ Pommerenke,
On iteration of analytic functions in a halfplane.
J.\ London Math.\ Soc.\ (2) 19 (1979), 439--447.

\bibitem{QVV}
S.\ L.\ Qiu, M.\ K.\ Vamanamurthy and M.\ Vuorinen,
Bounds for quasiconformal distortion functions.
J.\ Math.\ Anal.\ Appl.\ 205 (1997), 43-–64.

\bibitem{Rippon}
P. J. Rippon, Baker domains of meromorphic functions. 
Ergodic Theory Dynam. Systems 26 (2006), 1225--1233.

\bibitem{Rippon2008}
P.\ J.\ Rippon, Baker domains.
In ``Transcendental Dynamics and Complex Analysis''.
London Math.\ Soc.\ Lect.\ Note Ser.\ 348.
Edited by P.\ J.\ Rippon and G.\ M.\ Stallard,
Cambridge Univ.\ Press, Cambridge, 2008, 371--395.

\bibitem{RipponStallard} P. J. Rippon and G. M. Stallard, Families
of Baker domains II. Conform. Geom. Dyn. 3 (1999), 67--78.

\bibitem{Shishikura1987} 
M.\ Shishikura,
On the quasi-conformal surgery of rational functions. 
{\rm Ann.\ Sci.\ \'Ecole Norm.\ Sup.} (4)
20 (1987), 1--29.


\bibitem{Steinmetz} N.\ Steinmetz,
{\rm Rational iteration}.
Walter de Gruyter, Berlin
1993.

\bibitem{Vuorinen88}
M. Vuorinen,
Conformal geometry and quasiregular mappings.
Lecture Notes in Mathema\-tics 1319, Springer-Verlag, Berlin, 1988.

\bibitem{Zakeri} S. Zakeri, On Siegel disks of a class of entire
functions. Duke Math. J. 152 (2010), 481--532.

\bibitem{Zheng} J. H. Zheng, Iteration of functions meromorphic
outside a small set. Tohoku Math. J. 57 (2005), 23--43.

\bibitem{Zheng1} J. H. Zheng, Dynamics of transcendental meromorphic
functions (Chinese). Monograph of Tsinghua University, Tsinghua University
Press, Beijing, 2006.

\end{thebibliography}
\end{document}